\documentclass[a4paper,reqno]{amsart}
\usepackage{amsmath,amssymb,amsthm,graphicx,mathrsfs,url}
\usepackage[usenames,dvipsnames]{color}
\definecolor{darkred}{rgb}{0.4,0.1,0.1}
\definecolor{darkblue}{rgb}{0.1,0.1,0.4}

\usepackage{tikz}
\usetikzlibrary{datavisualization}
\usetikzlibrary{datavisualization.formats.functions}
\usepackage[normalem]{ulem}

\numberwithin{equation}{section}
\theoremstyle{plain}

\newtheorem{theorem}{Theorem}[section]
\newtheorem{lemma}[theorem]{Lemma}

\newtheorem{proposition}[theorem]{Proposition}
\newtheorem{corollary}[theorem]{Corollary}

\theoremstyle{remark}
\newtheorem{remark}[theorem]{Remark}

\theoremstyle{definition}
\newtheorem{example}[theorem]{Example}

\DeclareMathOperator{\spann}{span}
\DeclareMathOperator{\diver}{div}

\DeclareMathOperator{\Real}{Re}
\DeclareMathOperator{\Imag}{Im}

\definecolor{darkgreen}{rgb}{0.1,0.45,0.1}
\definecolor{darkblue}{rgb}{0.1,0.1,0.4}
\definecolor{darkgrey}{rgb}{0.5,0.5,0.5}

\definecolor{darkred}{rgb}{0.6,0.0,0.0}

\newcommand\void[1]{}

\def\sa{\mathfrak a}

\def\R{\mathbb{R}}
\def\C{\mathbb{C}}

\def\N{\mathbb{N}}

\renewcommand{\div}{\mathrm{div}\,}

\newcommand{\dom}{\mathrm{dom}\,}

\allowdisplaybreaks

\def\dd{{\,\mathrm d}}

\newcounter{counter_a}

\title[Inequalities between eigenvalues of Schr\"odinger operators]{Inequalities between Neumann and Dirichlet eigenvalues of Schr\"odinger operators} 

\author[J.~Rohleder]{Jonathan Rohleder}
\address{Matematiska institutionen \\ Stockholms universitet \\
106 91 Stockholm \\
Sweden}
\email{jonathan.rohleder@math.su.se}

\begin{document}

\begin{abstract}
Given a Schr\"odinger operator with a real-valued potential on a bounded, convex domain or a bounded interval we prove inequalities between the eigenvalues corresponding to Neumann and Dirichlet boundary conditions, respectively. The obtained inequalities depend partially on monotonicity and convexity properties of the potential. The results are counterparts of classical inequalities for the Laplacian but display some distinction between the one-dimensional case and higher dimensions.
\end{abstract}

\maketitle

\section{Introduction}

Let $\Omega$ be either a bounded interval or a bounded, convex domain in $\R^d$ and let $V : \Omega \to \R$ be a bounded potential. For the Schr\"odinger operator $- \Delta + V$ on $\Omega$ we denote by 
\begin{align*}
 \mu_1 (V) < \mu_2 (V) \leq \mu_3 (V) \leq \dots
\end{align*}
the eigenvalues corresponding to Neumann boundary conditions and by
\begin{align*}
 \lambda_1 (V) < \lambda_2 (V) \leq \lambda_3 (V) \leq \dots
\end{align*}
those corresponding to Dirichlet boundary conditions, taking multiplicities into account in both cases. The aim of this paper is to establish inequalities between these eigenvalues that improve the trivial estimate $\mu_k (V) \leq \lambda_k (V)$; the latter follows directly from the variational characterizations of the eigenvalues, and the inequality can actually be seen to be strict with the help of a unique continuation principle, see, e.g.,~\cite[Theorem~3.2]{BRS18}.

The case of the Laplacian, i.e.\ $V = 0$ identically, has been studied in depth. While for $d = 1$ a simple calculation gives 
\begin{align*}
 \mu_{k + 1} (0) = \frac{k^2 \pi^2}{L (\Omega)^2} = \lambda_k (0), \qquad k \in \N,
\end{align*}
where $L (\Omega)$ denotes the length of the interval $\Omega$, in dimensions $d \geq 2$ inequalities between Neumann and Dirichlet eigenvalues of the Laplacian are a classical topic in spectral theory with a long history. To name only a few steps in the development, P\'olya proved $\mu_2 (0) < \lambda_1 (0)$ on any sufficiently regular $\Omega$~\cite{P52}, while Payne showed $\mu_{k + 2} (0) < \lambda_k (0)$ for all $k$ on any convex, smooth domain in $d = 2$~\cite{P55}. This was generalized by Levine and Weinberger who established 
\begin{align}\label{eq:LW}
 \mu_{k + d} (0) \leq \lambda_k (0), \qquad k \in \N,
\end{align}
on any convex domain, with strict inequality for sufficiently smooth $\Omega$~\cite{LW86}. For non-convex $\Omega$ the best result known is $\mu_{k + 1} (0) < \lambda_k (0)$ due to Friedlander~\cite{F91} and Filonov~\cite{F05}, see also~\cite{AM12}. The question if~\eqref{eq:LW} extends to non-convex domains remains open and is a subject of current research; see, e.g., the recent preprint~\cite{CMS19}. Also similar questions where studied for the Laplacian on the Heisenberg group~\cite{FL10,H08} and on manifolds~\cite{AL97,M91}. Of course all mentioned estimates extend to Schr\"odinger operators with {\em constant} potentials since adding a constant $V_0$ simply shifts all Neumann and Dirichlet eigenvalues by~$V_0$.

The case of a non-constant potential $V$ has not received much attention yet and shall be considered here. We look first into the case $d = 1$. If the potential is symmetric with respect to the center of the interval, w.l.o.g.\ $\Omega = (- r, r)$, and is monotonous on each half-interval we show that
\begin{align}\label{eq:monotone}
 \begin{cases}
  \mu_2 (V) < \lambda_1 (V) & \text{if}~V~\text{is non-increasing on}~(0, r), \\
  \lambda_1 (V) < \mu_2 (V) & \text{if}~V~\text{is non-decreasing on}~(0, r),
 \end{cases}
\end{align}
as long as $V$ is not constant. In particular, if $V$ is symmetric then
\begin{align*}
 \begin{cases}
  \mu_2 (V) < \lambda_1 (V) & \text{if}~V~\text{is concave},\\
  \lambda_1 (V) < \mu_2 (V) & \text{if}~V~\text{is convex}.
 \end{cases}
\end{align*}
These statements do not hold in general for higher eigenvalues or non-symmetric cases as we show with the help of examples. However, they remain valid for non-symmetric convex respectively concave potentials that are sufficiently small perturbations of a constant, also for higher eigenvalues. The proof of~\eqref{eq:monotone} is based on the Hellmann--Feynman formula for the change of the eigenvalues under a perturbation of the potential. 

In dimension $d \geq 2$ the situation is slightly different and inequalities of the form $\mu_{k + r} (V) \leq \lambda_k (V)$ may hold for some positive $r$ even if $V$ is convex. In fact under the assumption that $V$ is weakly differentiable and satisfies 
\begin{align}\label{eq:orthogonal}
 \nabla V (x) \perp F \qquad \text{for a.a.}~x \in \Omega
\end{align}
for a subspace $F$ of $\R^d$ with $\dim F = r$ we prove
\begin{align}\label{eq:VorthogonalIntro}
 \mu_{k + r} (V) \leq \lambda_k (V), \qquad k \in \N.
\end{align}
The condition~\eqref{eq:orthogonal} means that $V$ depends only on $d - r$ directions. If $V$ is non-constant then the best possible estimate in~\eqref{eq:orthogonal} is 
\begin{align*}
 \mu_{k + d - 1} (V) \leq \lambda_k (V), \qquad k \in \N,
\end{align*}
which is true if $V$ is ``one-dimensional'', i.e. depends only on one variable (up to a change of coordinates); see Example~\ref{ex:dminus1} below. 

However, convexity properties influence the eigenvalue inequalities considered here also in dimensions $d \geq 2$. If the potential $V$ is concave and non-constant we show 
\begin{align}\label{eq:concave}
 \mu_d (V) < \lambda_1 (V)
\end{align}
without requiring the condition~\eqref{eq:orthogonal}. If $\Omega$ and $V$ are both symmetric with respect to all coordinate axes and $V$ is concave then we even get
\begin{align*}
 \mu_{d + 1} (V) < \lambda_1 (V)
\end{align*}
as long as $V$ is not constant. This estimate applies for instance if $\Omega$ is a ball and $V$ is concave and radially symmetric, a case that is not covered by~\eqref{eq:orthogonal}.

The proofs of our multidimensional results are variational. They are based on a set of test functions suggested for the Laplacian in~\cite{LW86} and on techniques developed for Laplacian eigenvalues on polyhedral domains in~\cite{LR17}; cf.\ also~\cite{R20}.

\section{Preliminaries}

Let us set the stage and collect a few well-known facts on the Schr\"odinger operators on bounded domains and their spectra. Let $\Omega \subset \R^d$, $d \geq 1$, be a bounded, convex domain; for $d = 1$ this reduces to a bounded, open interval. We denote by $H^k (\Omega)$ the usual Sobolev space of order $k$ on $\Omega$, $k = 1, 2, \dots$, and by $H_0^1 (\Omega)$ the closure of $C_0^\infty (\Omega)$ in $H^1 (\Omega)$. As any convex domain has a Lipschitz boundary, there is a well-defined trace operator $u \mapsto u |_{\partial \Omega}$ defined on $H^1 (\Omega)$ that acts as the restriction to the boundary for $u$ that are continuous on $\overline{\Omega}$; the space $H_0^1 (\Omega)$ coincides with the kernel of the trace operator. Moreover, we denote by $u \mapsto \partial_\nu u |_{\partial \Omega}$ the trace of the derivative of $u$ with respect to the outer unit normal on $\partial \Omega$, defined on $H^2 (\Omega)$ in a weak sense; see, e.g.,~\cite[Lemma~4.3]{McL}. 

The Neumann and Dirichlet Laplacians on $\Omega$ are defined as
\begin{align*}
 - \Delta_{\rm N} u = - \Delta u, \quad \dom \big( - \Delta_{\rm N} \big) = \big\{ u \in H^2 (\Omega) : \partial_\nu u |_{\partial \Omega} = 0 \big\}
\end{align*}
and
\begin{align*}
 - \Delta_{\rm D} u = - \Delta u, \quad \dom \big( - \Delta_{\rm D} \big) = \big\{ u \in H^2 (\Omega) : u |_{\partial \Omega} = 0 \big\},
\end{align*}
respectively; in the case $d = 1$, where $\Omega = (a, b)$ for some $a < b$, the condition $\partial_\nu u |_{\partial \Omega} = 0$ has to be interpreted accordingly as $u' (a) = u' (b) = 0$. These Laplacians are self-adjoint operators in $L^2 (\Omega)$ with purely discrete, non-negative spectra that accumulate to $+ \infty$; for more details we refer the reader to, e.g.,~\cite{EE87}.

Throughout this paper we assume that $V : \Omega \to \R$ is a measurable, bounded function. The Schr\"odinger operators $- \Delta_{\rm N} + V$ and $- \Delta_{\rm D} + V$ are then perturbations of the respective Laplacians by a bounded, self-adjoint multiplication operator and, hence, self-adjoint on the same domains. We denote the eigenvalues of $- \Delta_{\rm N} + V$ by
\begin{align*}
 \mu_1 (V) < \mu_2 (V) \leq \mu_3 (V) \leq \dots
\end{align*}
and the eigenvalues of $- \Delta_{\rm D} + V$ by
\begin{align*}
 \lambda_1 (V) < \lambda_2 (V) \leq \lambda_3 (V) \leq \dots,
\end{align*}
counted according to multiplicities. Note that the respective lowest eigenvalue has multiplicity one with a corresponding eigenfunction that can be chosen strictly positive inside $\Omega$, see, e.g.,~\cite{G84}, and that both operators may have negative eigenvalues as soon as $V$ has a non-trivial negative part. We will use the representation of the eigenvalues in terms of the min-max principles
\begin{align}\label{eq:minMaxNeumann}
 \mu_k (V) = \min_{\substack{L \subset H^1 (\Omega) \\ \dim L = k}} \,\, \max_{u \in L \setminus \{0\}} \frac{\sa [u]}{\int_\Omega |u|^2 \dd x}, \qquad k \in \N,
\end{align}
and
\begin{align}\label{eq:minmaxDirichlet}
 \lambda_k (V) = \min_{\substack{L \subset H_{0}^1 (\Omega) \\ \dim L = k}} \,\, \max_{u \in L \setminus \{0\}} \frac{\sa [u]}{\int_\Omega |u|^2 \dd x}, \qquad k \in \N,
\end{align}
where 
\begin{align}\label{eq:a}
 \sa [u] := \int_\Omega \big( |\nabla u|^2 + V |u|^2 \big) \dd x, \qquad u \in H^1 (\Omega),
\end{align}
is the quadratic form associated with $- \Delta + V$. In the case $d = 1$ we will also make use of the fact that the eigenvalues of $- \Delta_{\rm N} + \tau V$ and $- \Delta_{\rm D} + \tau V$ are analytic functions of the parameter $\tau \in \R$ and satisfy the Hellmann--Feynman formulae
\begin{align}\label{eq:HadamardNeumann}
 \frac{\dd}{\dd \tau} \mu_k (\tau V) = \int_0^\pi V (x) \psi^2 (x)
\end{align}
and
\begin{align}\label{eq:HadamardDirichlet}
 \frac{\dd}{\dd \tau} \lambda_k (\tau V) = \int_0^\pi V (x) \phi^2 (x),
\end{align}
where $\psi \in \ker (- \Delta_{\rm N} - \mu_k (\tau V))$ and $\phi \in \ker (- \Delta_{\rm D} - \lambda_k (\tau V))$ are $L^2$-normalized eigenfunctions; these formulae can, e.g., be derived from~\cite[Chapter~VII, equation~(3.18)]{Kato}.

\section{The one-dimensional case}\label{sec:1D}

In this section we assume $d = 1$, that is, $\Omega = (a, b)$ is a bounded, open interval. We start our investigation with the following observation that treats ``small'' convex (respectively concave) potentials.

\begin{proposition}\label{prop:calculate}
Let $V \in W^{2, \infty} (a, b)$ be real-valued, convex and non-constant and let $V_0 \in \R$. Then for each $k_0 \in \N$ there exists $\tau_0 > 0$ such that
\begin{align*}
 \begin{cases}
  \lambda_k (V_0 + \tau V) < \mu_{k + 1} (V_0 + \tau V) & \text{if}~\tau \in (0, \tau_0), \\
	\mu_{k + 1} (V_0 + \tau V) < \lambda_k (V_0 + \tau V) & \text{if}~\tau \in (- \tau_0, 0)
 \end{cases}
\end{align*}
holds for all $k \leq k_0$.
\end{proposition}

\begin{proof}
We may assume $V_0 = 0$ for simplicity. The statement is essentially a consequence of the Hellmann--Feynman formulae~\eqref{eq:HadamardNeumann} and~\eqref{eq:HadamardDirichlet} at $\tau = 0$. Indeed, assuming without loss of generality that $(a, b) = (0, \pi)$, we use the explicit expressions for the normalised Laplacian eigenfunctions and get through integration by parts
\begin{align*}
 \frac{\dd}{\dd \tau} \big( \lambda_k - \mu_{k + 1} \big) (\tau V) \Big|_{\tau = 0} & = \frac{2}{\pi} \int_0^\pi V (x) \big( \sin^2 (k x) - \cos^2 (k x) \big) \dd x \\
 & = \frac{2}{\pi} \int_0^\pi V' (x) \frac{\sin (2 k x)}{2 k} \dd x = - \frac{2}{\pi} \int_0^\pi V'' (x) \frac{\sin^2 (k x)}{2 k^2} \dd x.
\end{align*}
The latter expression is strictly negative as $V$ is convex and non-constant. Thus $\tau \mapsto \lambda_k (\tau V) - \mu_{k + 1} (\tau V)$ is strictly decreasing in a neighborhood of $\tau = 0$ and as $\lambda_k (0) - \mu_{k + 1} (0) = 0$ the assertion follows.
\end{proof}

Note that the case of negative $\tau$ in Proposition~\ref{prop:calculate} corresponds to a concave potential. Thus the statement indicates a change of the relation between $\lambda_k (V)$ and $\mu_{k + 1} (V)$ when $V$ changes from convex to concave (or vice versa). 

We will next derive a corresponding statement for the case $k = 1$ that is valid without any restrictions on the size of $V$. The core piece of its proof is the following lemma. Its statemtent is of an intuitive nature as the influence of a potential on the ground state energy should be smaller near a Dirichlet endpoint compared with a Neumann endpoint. However, we provide a rigorous proof.

\begin{lemma}\label{lem:mixed}
Let $V : (0, L) \to \R$ be measurable and bounded. Moreover, denote by $\lambda_1^{DN} (V)$ ($\lambda_1^{ND} (V)$, respectively) the lowest eigenvalue of the Schr\"odinger operator $- \frac{\text{d}^2}{\text{d} x^2} + V$ on $(0, L)$ subject to a Dirichlet condition at $0$ and a Neumann condition at $L$ (a Neumann condition at $0$ and a Dirichlet condition at $L$, respectively).
\begin{enumerate}
 \item If $V$ is non-increasing and non-constant then $\lambda_1^{DN} (V) < \lambda_1^{ND} (V)$.
 \item If $V$ is non-decreasing and non-constant then $\lambda_1^{ND} (V) < \lambda_1^{DN} (V)$.
\end{enumerate}
\end{lemma}

\begin{proof}
We prove~(i) only; the statement~(ii) follows from~(i) through reversion of the interval. We assume therefore that $V$ is non-increasing and non-constant. Let $\tau \in \R$ and let $\phi = \phi_{\rm DN}$ and $\psi = \psi_{\rm ND}$ be $L^2$-normalized eigenfunctions of the two mixed eigenvalue problems corresponding to $\lambda_1^{\rm DN} (\tau V)$ and $\lambda_1^{\rm ND} (\tau V)$, respectively; we may assume that $\psi (0) > 0$ and $\phi' (0) > 0$. Note that by standard Sturm--Liouville theory both $\phi$ and $\psi$ are positive inside $(0, L)$. Moreover, by the counterparts of~\eqref{eq:HadamardNeumann}--\eqref{eq:HadamardDirichlet} for mixed boundary conditions we have
\begin{align}\label{eq:HadamardDN}
 \frac{\dd}{\dd \tau} \lambda_1^{\rm DN} (\tau V) = \int_0^L V (x) \phi^2 (x) \dd x
\end{align}
and 
\begin{align}\label{eq:HadamardND}
 \frac{\dd}{\dd \tau} \lambda_1^{\rm ND} (\tau V) = \int_0^L V (x) \psi^2 (x) \dd x
\end{align}
for all $\tau \in \R$; recall that, although not explicitly indicated by our notation, $\phi$ and $\psi$ depend on $\tau$. 
 
Let us study the quotient $\psi / \phi$; the following argumentation is inspired by a reasoning from~\cite{L94}. The function $\psi / \phi$ defined on $(0, L)$ tends to $+ \infty$ at $0$ and is zero at $L$. We have
\begin{align*}
 \big( \lambda_1^{\rm DN} - \lambda_1^{\rm ND} \big) (\tau V) \int_0^x \phi \psi & = \int_0^x  (- \phi'') \psi + \tau \int_0^x V \phi \psi - \int_0^x \phi (- \psi'') - \tau \int_0^x V \phi \psi \\
& = \phi' (0) \psi (0) - \phi' (x) \psi (x) + \phi (x) \psi' (x)
\end{align*}
and, hence,
\begin{align}\label{eq:quotient}
 \Big( \frac{\psi}{\phi} \Big)' (x) = \frac{\psi' \phi - \psi \phi'}{\phi^2} (x) = \frac{\big( \lambda_1^{\rm DN} - \lambda_1^{\rm ND} \big) (\tau)}{\phi^2 (x)} \int_0^x \phi \psi - \frac{\phi' (0) \psi (0)}{\phi^2 (x)}
\end{align}
holds for any $x \in (0, L)$. As the integral on the right-hand side is positive for all $x \in (0, L)$ and $\phi' (0) \psi (0) > 0$ this implies that $\psi/\phi$ is strictly decreasing on $(0, L)$ whenever $\lambda_1^{\rm DN} (\tau V) - \lambda_1^{\rm ND} (\tau V) \leq 0$. Consequently, $\psi / \phi$ takes the value $1$ exactly once or, equivalently, $\phi^2 - \psi^2$ has exactly one zero $x_0$ in $(0, L)$ for such $\tau$.

Now we combine~\eqref{eq:HadamardDN} and~\eqref{eq:HadamardND} and get
\begin{align}\label{eq:theBigDeal}
 \frac{\dd}{\dd \tau} \big( \lambda_1^{\rm DN} - \lambda_1^{\rm ND} \big) (\tau V) & = \int_0^{x_0} V (\phi^2 - \psi^2) + \int_{x_0}^L V (\phi^2 - \psi^2).
\end{align}
By monotonicity of $V$ we have $V (x) \geq V (x_0)$ for $x \in (0, x_0)$ and $V (x) \leq V (x_0)$ on $(x_0, L)$. On the other hand, $\phi^2 - \psi^2$ is negative on $(0, x_0)$ and positive on $(x_0, L)$. Thus~\eqref{eq:theBigDeal} and the assumption that $V$ is non-constant lead to the estimate
\begin{align*}
 \frac{\dd}{\dd \tau} \big( \lambda_1^{\rm DN} - \lambda_1^{\rm ND} \big) (\tau V) & < V (x_0) \int_0^L (\phi^2 - \psi^2) = 0,
\end{align*}
whenever $\lambda_1^{\rm DN} (\tau V) - \lambda_1^{\rm ND} (\tau V) \leq 0$, where the last equality relies on $\phi$ and $\psi$ being $L^2$-normalized. As $\lambda_1^{\rm DN} (0) - \lambda_1^{\rm ND} (0) = 0$ there exists $\tau^* > 0$ such that $\frac{\dd}{\dd \tau} (\lambda_1^{\rm DN} - \lambda_1^{\rm ND}) (\tau V)$ is negative on $[0, \tau^*)$ and, consequently, $\lambda_1^{\rm DN} - \lambda_1^{\rm ND} (\tau V)$ is negative on $(0, \tau^*]$. But then $\frac{\dd}{\dd \tau} (\lambda_1^{\rm DN} - \lambda_1^{\rm ND}) (\tau V) |_{\tau = \tau^*} < 0$ and we can continue successively. By a compactness argument, we will arrive at the case $\tau = 1$ leading to $\lambda_1^{\rm DN} (V) - \lambda_1^{\rm ND} (V) < 0$, which is assertion~(i).
\end{proof}

In the following we consider the case $\Omega = (- r, r)$ for some $r > 0$ and a symmetric potential $V$, i.e.\ $V (- x) = V (x)$ for all $x \in (- r, r)$. Our main result in this section is the following.

\begin{theorem}\label{thm:main1}
Assume that $V : (- r, r) \to \R$ is bounded, measurable and symmetric. Then the following assertions hold.
\begin{enumerate}
 \item If $V$ is non-increasing and non-constant on $(0, r)$ then $\mu_2 (V) < \lambda_1 (V)$. 
 \item If $V$ is non-decreasing and non-constant on $(0, r)$ then $\lambda_1 (V) < \mu_2 (V)$. 
\end{enumerate}
\end{theorem}

\begin{proof}
Let us assume that $V$ satisfies the assumptions of the theorem and is non-constant. The eigenfunction corresponding to $\lambda_1 (V)$ is then positive inside $(- r, r)$, and it is an even function due to the symmetry of $V$. Similarly, the eigenfunction corresponding to $\mu_2 (V)$ is odd with its only zero at $0$. Hence in terms of the mixed eigenvalues on $(0, r)$ we have
\begin{align*}
 \mu_2 (V) = \lambda_1^{\rm DN} \big(V; (0, r) \big) \qquad \text{and} \qquad \lambda_1 (V) = \lambda_1^{\rm ND} \big(V; (0, r) \big),
\end{align*}
where we have used that eigenfunctions corresponding to higher eigenvalues of the mixed problems have zeroes inside $(0, r)$. Thus the assertions~(i) and~(ii) follow from Lemma~\ref{lem:mixed}~(i) and~(ii), respectively.
\end{proof}

Figure~\ref{fig:1dPotentials} displays examples of potentials to which Theorem~\ref{thm:main1} can be applied yielding different eigenvalue inequalities.
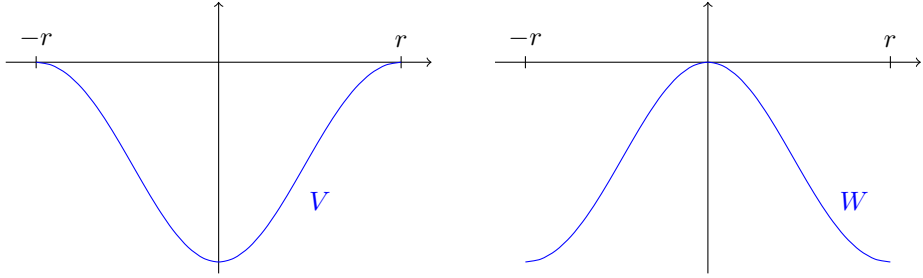
\begin{figure}[h]
    
\begin{tikzpicture}[scale=0.8]
 \draw[white] (2.0,-2.3) circle(0.00) node[left,blue]{$V$};
 \draw[->] (-3.5,0) -- (3.5,0);
 \draw[->] (0,-3.5) -- (0,1);
 \draw (-3,-0.1) -- (-3,0.1) node[above] {$-r$};
 \draw (3,-0.1) -- (3,0.1) node[above] {$r$};
 \draw[scale=3,domain=-1:1,smooth,variable=\x,blue] plot ({\x},{3*\x*\x*exp(-\x*\x)-3*exp(-1)});
\end{tikzpicture} \qquad
\begin{tikzpicture}[scale=0.8]
 \draw[white] (2.8,-2.3) circle(0.00) node[left,blue]{$W$};
 \draw[->] (-3.5,0) -- (3.5,0);
 \draw[->] (0,-3.5) -- (0,1);
 \draw (-3,-0.1) -- (-3,0.1) node[above] {$-r$};
 \draw (3,-0.1) -- (3,0.1) node[above] {$r$};
 \draw[scale=3,domain=-1:1,smooth,variable=\x,blue] plot ({\x},{- 3*\x*\x*exp(-\x*\x)});
\end{tikzpicture}

\caption{Two potentials satisfying the assumptions of Theorem~\ref{thm:main1} such that $\lambda_1 (V) < \mu_2 (V)$ and $\mu_2 (W) < \lambda_1 (W)$.}

\label{fig:1dPotentials}
 
\end{figure}

As any convex function $V$ being symmetric on $(- r, r)$ is non-decreasing on $(0, r)$, Theorem~\ref{thm:main1} has the following immediate implications for the case of a convex, respectively concave, potential. 

\begin{corollary}\label{cor:ConvexConcave}
Assume that $V : (- r, r) \to \R$ is bounded, measurable and symmetric. Then the following assertions hold.
\begin{enumerate}
 \item If $V$ is convex and non-constant then $\lambda_1 (V) < \mu_2 (V)$.
 \item If $V$ is concave and non-constant then $\mu_2 (V) < \lambda_1 (V)$.
\end{enumerate}
\end{corollary}

\begin{remark}
We would like to point out that the argumentation of Theorem~\ref{thm:main1} extends to higher eigenvalues if the potential $V$ has more symmetries. For instance,
\begin{align}\label{eq:higher}
 \mu_3 (V) = \mu_2 \big(V; (0, r) \big) \qquad \text{and} \qquad \lambda_2 (V) = \lambda_1 \big(V; (0, r) \big)
\end{align}
by the symmetry of $V$ with respect to the origin. If, in addition, $V$ is also symmetric with respect to the point $r/2$ within $(0, r)$ then one can apply Theorem~\ref{thm:main1} to the interval $(0, r)$ to obtain 
\begin{align*}
 \begin{cases}
  \lambda_2 (V) < \mu_3 (V) & \text{if $V$ is non-decreasing on}~(\frac{r}{2}, r), \\
  \mu_3 (V) < \lambda_2 (V) & \text{if $V$ is non-increasing on}~(\frac{r}{2}, r),
 \end{cases}
\end{align*}
assuming that $V$ is not constant. Analogous statements hold for higher eigenvalues. 
\end{remark}

The purpose of the next example is two-fold. On the one hand it shows that without symmetry of the potential either one or the other inequality between $\lambda_1 (V)$ and $\mu_2 (V)$ may hold. On the other hand, it shows that the statements of Theorem~\ref{thm:main1} do not carry over to higher eigenvalues.

\begin{example}\label{ex:step}
On the interval $(0, 2)$ consider the potential $V$ satisfying
\begin{align*}
 V (x) = \begin{cases}
  c & \text{for}~x \in (0, 1), \\
	0 & \text{for}~x \in (1, 2),
 \end{cases}
\end{align*}
where $c$ is a positive constant, see Figure~\ref{fig:stepPotential}. 
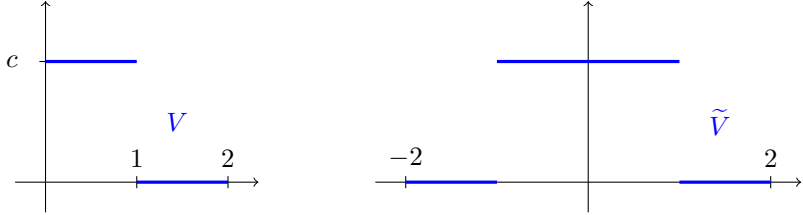
\begin{figure}[h]
    
\begin{tikzpicture}[scale=0.8]
 \draw[white] (2.5,1.0) circle(0.00) node[left,blue]{$V$};
 \draw (3,-0.1) -- (3,0.1) node[above] {$2$};
 \draw (1.5,-0.1) -- (1.5,0.1) node[above] {$1$};
 \draw (-0.1,2) -- (0.1,2);
 \draw (-0.3,2) circle(0.00) node[left]{$c$};
 \draw[->] (-0.5,0) -- (3.5,0); 
 \draw[->] (0,-0.5) -- (0,3);
 \draw[-,blue,very thick] (0,2) -- (1.5,2);
 \draw[-,blue,very thick] (1.5,0) -- (3,0);
\end{tikzpicture} \qquad \qquad
\begin{tikzpicture}[scale=0.8]
 \draw[white] (2.5,1.0) circle(0.00) node[left,blue]{$\widetilde V$};
 \draw (-3,-0.1) -- (-3,0.1) node[above] {$- 2$};
 \draw (3,-0.1) -- (3,0.1) node[above] {$2$};
 \draw (-0.1,2) -- (0.1,2);
 \draw[->] (-3.5,0) -- (3.5,0); 
 \draw[->] (0,-0.5) -- (0,3);
 \draw[-,blue,very thick] (-3,0) -- (-1.5,0);
 \draw[-,blue,very thick] (-1.5,2) -- (1.5,2);
 \draw[-,blue,very thick] (1.5,0) -- (3,0);
\end{tikzpicture}

\caption{The potentials $V$ and $\widetilde V$ in Example~\ref{ex:step}.}

\label{fig:stepPotential}
 
\end{figure}
An easy calculation yields that the Dirichlet eigenvalues $\lambda \neq c$ for the potential $V$ are given by the squares of the positive roots of the equation
\begin{align*}
 \sqrt{k^2 - c} \sin k \cos \sqrt{k^2 - c} + k \cos k \sin \sqrt{k^2 - c} = 0,
\end{align*}
and that $\lambda = c$ is an eigenvalue if and only if $\sqrt{c} + \tan \sqrt{c} = 0$. Similarly, the Neumann eigenvalues correspond to the roots of
\begin{align*}
 \sqrt{k^2 - c} \sin \sqrt{k^2 - c} \cos k + k \sin k \cos \sqrt{k^2 - c} = 0.
\end{align*}
The numerical approximate values of the lowest respective roots for several choices of $c$ are displayed in Table~\ref{tab:numerics}.
\begin{table}[h]%
\begin{tabular}{|c|c|c|c|c|c|c|}
 \hline
 $c$ & $10^{-4}$ & $10^{-2}$ & $10^{-1}$ & 1 & 10 & 100 \\
 \hline
 $\sqrt{\mu_2 (V)}$ & 1.5708 & 1.572389 & 1.5869 & 1.7438 & 3.1553 & 4.2711 \\
 \hline
 $\sqrt{\lambda_1 (V)}$ & 1.5724 & 1.572386 & 1.5866 & 1.7153 & 4.0270 & 10.3793 \\
 \hline
\end{tabular}
\vspace*{2mm}
\caption{Approximate values of the square roots of $\mu_2 (V)$ and $\lambda_1 (V)$ for the step potential $V$ in Example~\ref{ex:step}.}
\label{tab:numerics}
\end{table}
These values show that $\mu_2 (V) < \lambda_1 (V)$ holds for $c$ close to zero (which can also be seen with the technique of Proposition~\ref{prop:calculate}), and they indicate the same for large values of $c$. However, in between there exist values of $c$ for which $\lambda_1 (V) < \mu_2 (V)$ holds.

Consider, moreover, the interval $(- 2, 2)$ and the step potential $\widetilde V$ obtained from the above step potential $V$ by reflecting it symmetrically to $(- 2, 0)$; cf.\ Figure~\ref{fig:stepPotential}. Then $\widetilde V$ is non-increasing on $(0, 2)$. Applying~\eqref{eq:higher} in, e.g., the case $c = 10^{- 4}$ we get 
\begin{align*}
 \mu_3 (\widetilde V) = \mu_2 \big( V; (0, 2) \big) < \lambda_1 \big(V; (0, 2) \big) = \lambda_2 (\widetilde V),
\end{align*}
while for $c = 1$ the same reasoning implies $\lambda_2 (\widetilde V) < \mu_3 (\widetilde V)$. This shows that Theorem~\ref{thm:main1} does not extend to higher eigenvalues in general.
\end{example}

\section{Multidimensional results}\label{sec:multy}

In this section we study the multidimensional case and assume that $\Omega$ is a bounded, convex domain in $\R^d$ for some $d \geq 2$. 

First, it is an immediate consequence of the known inequalities for the Laplacian that certain nontrivial inequalities also hold for Schr\"odinger operators with potentials that are sufficiently close to constants. More specifically, if, for instance, the boundary of $\Omega$ has H\"older continuous second derivatives then the eigenvalues of the Laplacian or any Schr\"odinger operator with a constant potential $V_0$ satisfy 
\begin{align*}
 \mu_{k + d} (V_0) < \lambda_k (V_0)
\end{align*}
for all $k \in \N$ by~\cite[Theorem~2.1]{LW86}. Consequently, if $V \in L^\infty (\Omega)$ is any real-valued potential then for each $k_0 \in \N$ there exists $\tau_0 > 0$ such that 
\begin{align*}
 \mu_{k + d} (V_0 + \tau V) < \lambda_k (V_0 + \tau V)
\end{align*}
holds for all $k \leq k_0$ and all $\tau \in \R$ with $|\tau| < \tau_0$. This is true without any convexity assumption on $V$ and is therefore in contrast to the one-dimensional case. A similar observation can be obtained on any bounded, not necessarily convex Lipschitz domain: there one has at least $\mu_{k + 1} (V_0 + \tau V) < \lambda_k (V_0 + \tau V)$ for all sufficiently small $k$ and $|\tau|$ as a consequence of the inequality $\mu_{k + 1} (0) < \lambda_k (0)$ in~\cite{F05}.

For concrete domains this observation can be improved and quantified as the following example shows.

\begin{example}
Let $\Omega = (0, \pi) \times (0, \pi) \subset \R^2$. Then the Neumann and Dirichlet spectra of the Laplacian are simply given by
\begin{align*}
 \sigma (- \Delta_{\rm N}) = \big\{ m^2 + n^2 : m, n \in \N_0 \big\} \quad \text{and}  \quad \sigma (- \Delta_{\rm D}) = \{m^2 + n^2 : m, n \in \N \big\}.
\end{align*}
Counted with multiplicities, the first Neumann eigenvalues are $0, 1, 1, 2, 4, 4, 5, 5$, $8, 9, \dots$ and the first Dirichlet eigenvalues are $2, 5, 5, 8, 10, 10$, $13, 13$, $17, 17, \dots$; one can see easily that $\lambda_k (0) - \mu_{k + 2} (0) \geq 1$ holds for all $k \in \N$. Hence, if $V$ is for instance a measurable, non-negative potential and $V_0$ is constant then
\begin{align*}
 \mu_{k + 2} (V_0 + \tau V) \leq \mu_{k + 2} (V_0) + 1 \leq \lambda_k (V_0) \leq \lambda_k (V_0 + \tau V)
\end{align*}
holds for all $k \in \N$ and for all $\tau \in [0, \tau_0]$, where $\tau_0 := \|V\|_\infty^{- 1}$.
\end{example}

In general the above observation is not quantitative and therefore of limited practical use. In the following main theorem of this section we establish an inequality that is independent of the strength of the potential. To compensate for this, we assume $V$ to be constant in some directions (up to a change of coordinates). We point out that we do not make any convexity assumptions on the potential here.

\begin{theorem}\label{thm:allEigenvaluesND}
Let $\Omega \subset \R^d$, $d \geq 2$, be a bounded, convex domain and let $V \in W^{1, \infty} (\Omega)$ be real-valued. Assume that there exists an $r$-dimensional subspace $F$ of~$\R^d$ such that $\nabla V (x)$ is orthogonal to $F$ for almost all $x \in \Omega$. Then
\begin{align*}
 \mu_{k + r} (V) \leq \lambda_k (V)
\end{align*}
holds for all $k \in \N$.
\end{theorem}

\begin{proof}
We show the statement first for the case that $\Omega$ is a polyhedral domain; this means that the boundary of $\Omega$ is piecewise flat, consisting of $(d - 1)$-dimensional hyperplanes. Afterwards we will derive the result for general convex domains by approximation. Working with polyhedral $\Omega$ allows us to make use of the identity
\begin{align}\label{eq:Grisvard}
 \int_\Omega (\partial_{m l} u) (\partial_{m j} u) \dd x = \int_\Omega (\partial_{m m} u) (\partial_{l j} u) \dd x, \qquad u \in H_0^1 (\Omega) \cap H^2 (\Omega),
\end{align}
that is valid for all $m, l, j \in \{1, \dots, d\}$; see~\cite[Lemma~4.3.1.1--Lemma~4.3.1.3]{G85} for the two-dimensional case of a polygon and~\cite[Lemma~A.1]{LR17} for higher dimensions.\footnote{The integral identity~\eqref{eq:Grisvard} and its proof given in~\cite{LR17} do not actually require $\Omega$ to be convex (although stated for convex polyhedral domains in~\cite{LR17}). However, it is crucial for this identity that the boundary of $\Omega$ is piecewise flat; otherwise there are simple counterexamples as, e.g., the function $u (x, y) = 1 - x^2 - y^2$ on the unit disk in $\R^2$.}.  

We fix $k \in \N$ and choose an orthonormal family of real-valued eigenfunctions $u_j$ of $- \Delta_{\rm D} + V$ corresponding to the eigenvalues $\lambda_j (V)$, $j = 1, \dots, k$. Note that, as $\Omega$ is convex, $u \in H^2 (\Omega)$. Next we define the functions
\begin{align*}
 \Phi = \sum_{j = 1}^k a_j u_j \in H_0^1 (\Omega) \qquad \text{and} \qquad \Psi = b^\top \nabla u_k \in H^1 (\Omega),
\end{align*}
where $a_1, \dots, a_k$ are arbitrary complex numbers and $b = (b_1, \dots, b_d)^\top$ is any complex vector; such test functions were used in~\cite{LW86} in the case of the Laplacian. For the quadratic form $\sa$ in~\eqref{eq:a} we get
\begin{align}\label{eq:sumInForm}
 \sa [\Phi + \Psi] & = \sa [\Phi] + \sa [\Psi] + 2 \Real \int_\Omega \big(\nabla \Phi \cdot \overline{\nabla \Psi} + V \Phi \overline{\Psi} \big) \dd x.
\end{align}
It is our aim to evaluate the three summands on the right-hand side of~\eqref{eq:sumInForm}. First,
\begin{align}\label{eq:Phi}
\begin{split}
 \sa [\Phi] & = \sum_{l, j = 1}^k a_l \overline{a_j} \int_\Omega \big( (- \Delta u_l) u_j + V u_l u_j \big) \dd x \\
 & = \sum_{l, j = 1}^k \lambda_l (V) a_l \overline{a_j} \int_\Omega u_l u_j \dd x \\
 & = \sum_{j = 1}^k \lambda_j (V) |a_j|^2 \int_\Omega |u_j|^2 \dd x \leq \lambda_k (V) \int_\Omega |\Phi|^2 \dd x
\end{split}
\end{align}
due to the orthogonality of the $u_j$. Furthermore,
\begin{align}\label{eq:Psi}
\begin{split}
 \sa [\Psi] & = \sum_{m = 1}^d \int_\Omega \sum_{l = 1}^d b_l \partial_{m l} u_k \sum_{j = 1}^d \overline{b_j} \partial_{m j} u_k \dd x + \int_\Omega V b^\top \nabla u_k b^* \nabla u_k \dd x \\
 & = \sum_{m = 1}^d \int_\Omega (\partial_{m m} u_k) \sum_{l, j = 1}^d b_l \overline{b_j} \partial_{l j} u_k \dd x + \int_\Omega b^\top \nabla (V u_k)  b^* \nabla u_k \dd x \\
 & \qquad - \int_\Omega u_k b^\top \nabla V b^* \nabla u_k \dd x \\
 & = \int_\Omega \Delta u_k \diver \big( b b^* \nabla u_k \big) \dd x - \int_\Omega V u_k \diver \big( b b^* \nabla u_k \big) \dd x \\
 & \qquad - \int_\Omega u_k b^\top \nabla V b^* \nabla u_k \dd x \\
 & = - \lambda_k (V) \int_\Omega u_k \diver \big( b b^* \nabla u_k \big) \dd x - \int_\Omega u_k b^\top \nabla V b^* \nabla u_k \dd x \\
 & = \lambda_k (V) \int_\Omega \nabla u_k \cdot b b^* \nabla u_k \dd x - \int_\Omega u_k b^\top \nabla V b^* \nabla u_k \dd x \\
 & = \lambda_k (V) \int_\Omega |\Psi|^2 \dd x - \int_\Omega u_k b^\top \nabla V b^* \nabla u_k \dd x,
\end{split}
\end{align}
where we have used~\eqref{eq:Grisvard}. To treat the third summand in~\eqref{eq:sumInForm} we observe that
\begin{align*}
 - \Delta \Psi + V \Psi = b^\top \nabla (- \Delta u_k + V u_k) - u_k b^\top \nabla V = \lambda_k (V) \Psi - u_k b^\top \nabla V
\end{align*}
holds in the distributional sense and get
\begin{align}\label{eq:mixedTerm}
 \int_\Omega \big(\nabla \Phi \cdot \overline{\nabla \Psi} + V \Phi \overline{\Psi} \big) \dd x & = \lambda_k (V) \int_\Omega \Phi \overline{\Psi} \dd x - \int_\Omega \Phi u_k b^* \nabla V \dd x.
\end{align}
Plugging~\eqref{eq:Phi},~\eqref{eq:Psi} and~\eqref{eq:mixedTerm} into~\eqref{eq:sumInForm} yields
\begin{align}\label{eq:naAlso}
 \sa [\Phi + \Psi] \leq \lambda_k (V) \int_\Omega |\Phi + \Psi|^2 \dd x - 2 \Real \int_\Omega \Phi u_k b^* \nabla V \dd x - \int_\Omega u_k b^\top \nabla V b^* \nabla u_k \dd x.
\end{align}
Based on a unique continuation argument one can show as in the proof of~\cite[Theorem~4.1]{LR17} that the functions of the form $\Phi + \Psi$ given as above form a $(k + d)$-dimensional subspace of $H^1 (\Omega)$. When we restrict ourselves to such vectors $b$ such that $\Real b, \Imag b \in F$ in the definition of $\Psi$ then we get from~\eqref{eq:naAlso}
\begin{align*}
 \sa [\Phi + \Psi] \leq \lambda_k (V) \int_\Omega |\Phi + \Psi|^2 \dd x,
\end{align*}
where the functions $\Phi + \Psi$ form a $(k + r)$-dimensional subspace of $H^1 (\Omega)$. Hence the assertion of the theorem for any polyhedral domain follows by~\eqref{eq:minMaxNeumann}.

If $\Omega$ is an arbitrary convex domain then by piecewise linear interpolation of the boundary we can construct a sequence of polyhedral domains $(\Omega_n)_n$ which, due to the convexity of $\Omega$, are themselves convex and contained in $\Omega$ and approximate $\Omega$ in a sufficiently regular manner to obtain convergence of all Dirichlet and Neumann eigenvalues of the Schr\"odinger operator, cf.~\cite[Theorem~VI.10]{CH} for the case $d = 2$; the general case can be treated in the same way.
\end{proof}

\begin{remark}
\begin{enumerate}
 \item We point out that in the above theorem the best possible value of $r$ that can be reached by a non-constant potential is $d - 1$. On the other hand, for an arbitrary potential the theorem can be applied with $r = 0$ but this only yields the trivial estimate $\mu_k (V) \leq \lambda_k (V)$. Especially, in dimension $d = 1$ the reasoning of Theorem~\ref{thm:allEigenvaluesND} can only yield trivial estimates; cf.~Section~\ref{sec:1D}. For constant potentials in dimension $d \geq 2$, Theorem~\ref{thm:allEigenvaluesND} recovers the estimate~\eqref{eq:LW} due to Levine and Weinberger~\cite{LW86}.
 \item In the proof of Theorem~\ref{thm:allEigenvaluesND}, the convexity assumption on the domain $\Omega$ enters only in the $H^2$-regularity of the eigenfunctions of the Schr\"odinger operator with Dirichlet boundary conditions, which ensures that its derivatives are suitable as test functions, i.e.\ belong to $H^1 (\Omega)$. However, the question to what extent eigenvalue inequalities as those obtained here carry over to non-convex domains is open even for the case of potential zero; see, e.g., the recent discussion in~\cite[Section~1.2]{CMS19}.
\end{enumerate}
\end{remark}

We point out that in the above theorem the best possible value of $r$ that can be reached by a non-constant potential is $d - 1$. On the other hand, for an arbitrary potential the theorem can be applied with $r = 0$ but this only yields the trivial estimate $\mu_k (V) \leq \lambda_k (V)$. Especially, in dimension $d = 1$ the reasoning of Theorem~\ref{thm:allEigenvaluesND} can only yield trivial estimates; cf.~Section~\ref{sec:1D}.

In the following example Theorem~\ref{thm:allEigenvaluesND} can be applied with $r = d - 1$.

\begin{example}\label{ex:dminus1}
For $a, b \in \R \setminus \{0\}$ consider the potential
\begin{align*}
 V (x_1, \dots, x_d) = a e^{b (x_1 + \dots + x_d)}
\end{align*}
on any bounded, convex domain in $\R^d$. Then all partial derivatives of $V$ of first order equal $b V (x_1, \dots, x_d)$ and all partial derivatives of second order equal $b^2 V (x_1, \dots, x_d)$. In particular, $V$ may either be convex or concave, depending on the sign of $a$. Moreover, $\nabla V$ is contained in $\spann \{ V b (1, \dots, 1)^\top \}$ and Theorem~\ref{thm:allEigenvaluesND} yields
\begin{align*}
 \mu_{k + d - 1} (V) \leq \lambda_k (V)
\end{align*}
for all $k \in \N$.
\end{example}

Next we assume in addition that $V$ is concave, that is, the Hessian matrix $H_V$ of $V$ is negative semi-definite almost everywhere in $\Omega$. In this case one can prove a variant of Theorem~\ref{thm:allEigenvaluesND} for $k = 1$ without the gradient requirement on $V$.

\begin{theorem}\label{thm:generalLambdaOne}
If $\Omega \subset \R^d$, $d \geq 2$, is a bounded, convex domain and $V \in W^{2, \infty} (\Omega)$ is concave and real-valued then 
\begin{align*}
 \mu_d (V) \leq \lambda_1 (V)
\end{align*}
holds. If, in addition, $H_V (x)$ is negative definite on a subset of $\Omega$ with non-zero measure then $\mu_d (V) < \lambda_1 (V)$.
\end{theorem}

\begin{proof}
Let $k = 1$ and define $\Psi$ as in the proof of Theorem~\ref{thm:allEigenvaluesND} with arbitrary $b \in \C^d$. Then the calculation~\eqref{eq:Psi} is valid and can be rewritten as
\begin{align}\label{eq:erstmal}
 \sa [\Psi] & = \lambda_1 (V) \int_\Omega |\Psi|^2 \dd x - \int_\Omega \nabla V \cdot u_1 b b^* \nabla u_1 \dd x.
\end{align}
Moreover,
\begin{align}\label{eq:undDann}
\begin{split}
 \int_\Omega \nabla V \cdot u_1 b b^* \nabla u_1 \dd x & = \frac{1}{2} \int_\Omega \nabla V \cdot b b^* \nabla \big( u_1^2 \big) \dd x = - \frac{1}{2} \int_\Omega \div \big(b b^* \nabla V \big) u_1^2 \dd x \\
 & = - \frac{1}{2} \int_\Omega b^* H_V b u_1^2 \dd x \geq 0
\end{split}
\end{align}
as $V$ is concave; strict inequality holds if and only if $H_V$ is negative definite on a set of positive measure since $u_1^2$ is positive on $\Omega$. As the admitted functions $\Psi$ span a $d$-dimensional subspace of $H^1 (\Omega)$, the claim follows from combining~\eqref{eq:erstmal} and~\eqref{eq:undDann}.
\end{proof}

A further improvement of Theorem~\ref{thm:generalLambdaOne} can be shown if domain and potential are both symmetric. To be more specific, if $\Omega$ has $d$ axes of symmetry, without loss of generality $\Omega$ is symmetric with respect to all coordinate axes, and $V$ is symmetric with respect to all axes, i.e.\ 
\begin{align*}
 V (x_1, \dots, x_{j - 1}, x_j, x_{j + 1}, \dots, x_d) = V (x_1, \dots, x_{j - 1}, - x_j, x_{j + 1}, \dots, x_d),
\end{align*}
$j = 1, \dots, d$, the following assertion holds; it is in line with the one-dimensional Corollary~\ref{cor:ConvexConcave}~(ii).

\begin{theorem}\label{thm:multidimSymmetric}
Assume that $\Omega \subset \R^d$, $d \geq 2$, is a bounded, convex domain that is symmetric with respect to each coordinate axis and that $V \in W^{2, \infty} (\Omega)$ is concave, real-valued, and symmetric with respect to each variable. Then 
\begin{align*}
 \mu_{d + 1} (V) \leq \lambda_1 (V)
\end{align*}
holds. If, in addition, $H_V (x)$ is negative definite on a subset of $\Omega$ with non-zero measure then $\mu_{d + 1} (V) < \lambda_1 (V)$.
\end{theorem}

\begin{proof}
The proof is the same as for Theorem~\ref{thm:generalLambdaOne}. However, due to the symmetry assumptions on $\Omega$ and $V$, the Dirichlet eigenfunction $u_1$ corresponding to the lowest eigenvalue $\lambda_1 (V)$ is even (as it can be chosen strictly positive, see~\cite{G84}) with respect to each axis of symmetry. Thus $\partial_j u_1$ is odd with respect to the $j$-th coordinate axis and even with respect to all other axes and it follows 
\begin{align*}
 \int_\Omega v_1 \partial_j u_1 \dd x = 0, \qquad j = 1, \dots, d,
\end{align*}
for any Neumann eigenfunction $v_1$ corresponding to the lowest eigenvalue $\mu_1 (V)$ since $v_1$ is even with respect to all coordinates as well. Hence each possible test function $\Psi$ is orthogonal to $v_1$ and the claim follows.
\end{proof}

Examples of convex domains in $\R^2$ being symmetric with respect to $x$ and $y$ are displayed in Figure~\ref{fig:symmetric}. 
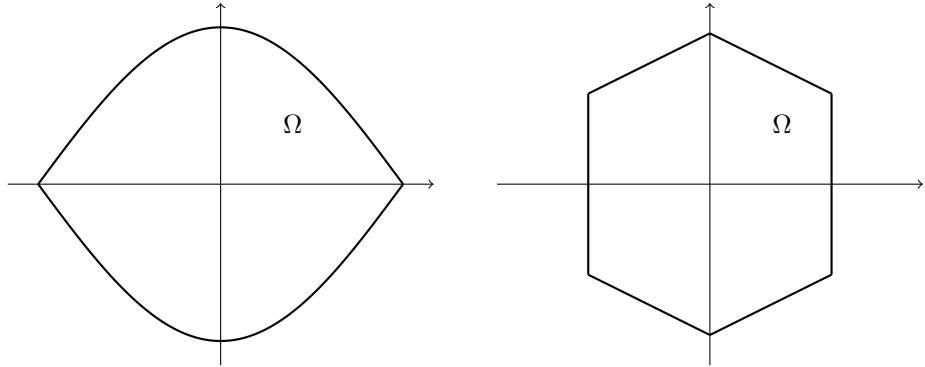
\begin{figure}[h]
    
\begin{tikzpicture}[scale=0.8]
 \draw[white] (1.5,1.0) circle(0.00) node[left,black]{$\Omega$};
 \draw[->] (-3.5,0) -- (3.5,0); 
 \draw[->] (0,-3) -- (0,3);
 \draw[scale=1,domain=0:3,smooth,variable=\x,thick] plot ({\x},{2.6*cos(30*\x)});
 \draw[scale=1,domain=-3:0,smooth,variable=\x,thick] plot ({\x},{2.6*cos(30*\x)});
 \draw[scale=1,domain=-3:0,smooth,variable=\x,thick] plot ({\x},{-2.6*cos(30*\x)});
 \draw[scale=1,domain=0:3,smooth,variable=\x,thick] plot ({\x},{-2.6*cos(30*\x)});
\end{tikzpicture} \qquad
\begin{tikzpicture}[scale=0.8]
 \draw[white] (1.5,1.0) circle(0.00) node[left,black]{$\Omega$};
 \draw[->] (-3.5,0) -- (3.5,0); 
 \draw[->] (0,-3) -- (0,3);
 \draw[thick] (-2,-1.5) -- (-2,1.5);
 \draw[thick] (2,-1.5) -- (2,1.5);
 \draw[thick] (-2,1.5) -- (0,2.5);
 \draw[thick] (2,1.5) -- (0,2.5);
 \draw[thick] (-2,-1.5) -- (0,-2.5);
 \draw[thick] (2,-1.5) -- (0,-2.5);
\end{tikzpicture}

\caption{Convex domains in $\R^2$ that are symmetric with respect to both $x$ and $y$ as required in Theorem~\ref{thm:multidimSymmetric}.}

\label{fig:symmetric}
 
\end{figure}
An example of a concave potential symmetric with respect to both $x$ and $y$ is the function
\begin{align*}
 V (x, y) = - c e^{x^2 + y^2}
\end{align*}
with a positive constant $c$. For this potential on any bounded, convex, symmetric domain Theorem~\ref{thm:multidimSymmetric} yields $\mu_3 (V) < \lambda_1 (V)$.

\section*{Acknowledgements}
The author gratefully acknowledges financial support by the grant no.\ 2018-04560 of the Swedish Research Council (VR). Moreover, the author wishes to express his gratitude to the anonymous referee for a very careful reading and for suggestions that helped to improve the presentation of the results.


\end{document}